\documentclass[10pt]{amsart}

\usepackage{amsmath,graphics}
\usepackage{amsfonts,amssymb}
\usepackage{amsmath,amscd}
\usepackage[OT2,T1]{fontenc}
\DeclareMathAlphabet{\mathcalligra}{T1}{calligra}{m}{n}
\DeclareSymbolFont{cyrletters}{OT2}{wncyr}{m}{n}
\DeclareMathSymbol{\Sha}{\mathalpha}{cyrletters}{"58}

\theoremstyle{plain}
\newtheorem*{theorem*}{Theorem}
\newtheorem*{lemma*} {Lemma}
\newtheorem*{corollary*} {Corollary}
\newtheorem*{proposition*} {Proposition}
\newtheorem{theorem}{Theorem}[section]

\newtheorem{proposition}[theorem]{Proposition}

\theoremstyle{remark}

\newtheorem*{definition}{Definition}

\theoremstyle{definition}

\def\part{\partial}

\def\bp{\begin{pmatrix}}

\def\ep{\end{pmatrix}}
\def\bn{\begin{enumerate}}

\def\en{\end{enumerate}}
\def\ba{\begin{array}}
\def\ea{\end{array}}

\def\fr12{\frac{1}{2}}

\def\cmtbf#1{} \def\cmt#1{}

\begin{document}

\title{A short note on p-adic families of Hilbert Modular Forms}
\author{Aftab Pande}

\begin{abstract}
We extend previous work of the author using an idea of Buzzard and give an elementary construction of non-ordinary $p$-adic families of Hilbert Modular Eigenforms. 
\end{abstract}

\maketitle

\section{Introduction}

The notion of $p$-adic analytic families of modular forms started with Serre \cite{S} using $p$-adic Eisenstein series. Hida (\cite{H2}, \cite{H3}) showed examples of cuspidal eigenforms of slope zero or ordinary cuspidal eigenforms. For non-ordinary cuspidal eigenforms, the conjectures of Gouvea-Mazur \cite{GM} asserted local constancy of the dimensions of the slope spaces. Then, Coleman \cite{C} showed that almost every overconvergent eigenform of finite slope lives in a $p$-adic family; these results were generalized by Coleman-Mazur \cite{CM} into a geometric object called the
eigencurve which is a rigid-analytic curve whose points correspond to
normalized finite-slope $p$-adic overconvergent modular eigenforms
of a fixed tame level $N$. In the case of Hilbert modular forms, Kisin-Lai \cite{KL} extended the construction of the eigencurve and showed that a finite slope Hilbert modular eigenform can be deformed into a one parameter family of finite slope eigenforms.

In this short note, we use an idea of Buzzard \cite{B} and give an elementary construction of a $p$-adic family of Hilbert Modular eigenforms. In previous work of the author \cite{P}, results on local constancy of slope $\alpha$ spaces of Hilbert modular forms were obtained in the spirit of the Gouvea-Mazur conjectures. Using the same setting as in \cite{P}, and assuming the dimensions of the slope $\alpha$ spaces is $1$, we are able to obtain a $p$-adic family of Hilbert Modular Forms. 

\begin{theorem}

Let $D(\mathbf{k},\alpha)$ be the number of eigenvalues of slope $\alpha$ of the $U_p$ operator acting on  $\mathcal{S}^{D}_{\mathbf{k}}(U,R)$, the space of Hilbert modular forms of weight $\mathbf{k}$. Assuming that the slope $\alpha$ spaces have dimension $1$ with $F_\mathbf{k}$ the unique eigenform of slope $\alpha$ and $a_\mathbf{k}(t)$ the eigenvalue of $t$ on $F_\mathbf{k}$ with $\kappa = \lfloor  c_1 n^{1/d+1} -1 - 3 \alpha \rfloor$, if $n > (\frac{\kappa + 1 + 3\alpha}{c_1})^{d+1}$and
$\mathbf{k},\mathbf{k'}$ sufficiently large are congruent to $\mathbf{k}_0$ and $\mathbf{k} \equiv \mathbf{k'} \mod p^{n -1}$  we have for all $t \in \mathbb{T}$, $a_{\mathbf{k}} (t) \equiv a_{\mathbf{k'}} (t) \mod p^{\kappa}$.

\end{theorem}

\section{Overview}
\begin{definition}
Let $c \in \mathbb{Z}_p$ and for $r \geq 0$, let $B(c,r) = \{ k \in
\mathbb{Z}_p : |k - c| < r \}$. Let $N$ be an integer prime to $p$.
Then a $p$-adic family of modular forms of level $N$ is a formal
power series:

\center{$\sum_{n \geq 0} F_n q^n$},

where each $F_n : B(c,r) \rightarrow \mathbb{C}_p$ is a $p$-adic
analytic function, with the property that for all sufficiently large (rational) integers $k$, each $\sum F_n(k) q^n$ is the
Fourier expansion of a modular form of weight $k$.
\end{definition}

The $p$-adic Eisenstein series $E_k^*(z) = E_k(z) -
p^{k-1}E_k(pz)$ are an example of a $p$-adic family of non-cuspidal eigenforms. 

After Hida's work on slope zero (ordinary) eigenforms Gouvea and Mazur \cite{GM} made some very precise conjectures
about the dimensions of the non-zero slope $\alpha$ spaces. Let $d(k,\alpha)$ be the
dimension of the slope $\alpha$ subspace of the space of classical
cuspidal eigenforms of weight $k$ for the $T_p$ operators. If $k_1 , k_2 > 2\alpha + 2$,
  and $k_1 \equiv k_2 \mod p^n(p-1)$
 then, $d(k_1, \alpha) = d(k_2, \alpha)$ (this condition is called local constancy). Buzzard and Calegari \cite{BC} later showed the conjecture is not true.

In \cite{P}, we obtained results on local constancy for Hilbert modular forms by finding bounds for the Newton polygons using methods of Buzzard \cite{B}. We give a brief description of Newton polygons here.

Let $L$ be a finite free $\mathbb{Z}_p$-module equipped with a
$\mathbb{Z}_p$-linear endomorphism $\xi$ and $\sum_{s=o}^{t} c_{s}X^{t-s}$ be the
characteristic polynomial of $\xi$ acting on $L \otimes \mathbb{Q}_{p}$.
Let $C$ denote the convex hull of the points $(i, v_{p}(c_{i}))$ in $\mathbb{R}^{2}$, for $0 \leq i \leq t$,
ignoring the $i$ for which $c_{i}= 0$. The Newton polygon of $\xi$ on $L$ is the lower
faces of $C$, that is the union of the sides forming the lower of
the two routes from $(0,0)$ to $(t, v_{p}(c_{t}))$ on the boundary
of $C$. If the Newton polygon has a
side of slope $\alpha$ whose projection onto the $x$ axis has
length $n$, then there are precisely $n$ eigenvectors of $\xi$ with
$p$-adic valuation equal to $\alpha$.

Our strategy is as follows:

$L$ will correspond to a space of automorphic forms and we will define $K$ to be a submodule of $L$
such that $L/K \equiv \oplus O/p^{a_{i}}O$ with the $a_{i}$
decreasing and $a_{i} \leq n$, where $n$ is fixed. We consider the characteristic polynomial $p(x)$ of $\xi$ (which corresponds to the $U_p$ operator) acting on $L$ and plot
its Newton polygon. We let $L'$ be a space of forms corresponding to a different weight, and choose $K'$ to be a submodule similar to $K$ such that modulo a certain power of $p$ the
spaces $L/K$ and $L'/K'$are isomorphic. This isomorphism leads to congruences of the
coefficients of the respective characteristic polynomials. 
This tells us that the Newton polygons of fixed slope coincide which gives us
local constancy of the slope $\alpha$ spaces. When the dimension is assumed to be one, there exists a unique eigenform $f \in L$ for $U_p$ with eigenvalue of slope $\alpha$. Using a result of Buzzard we show that these $f$'s form a family of eigenforms.

\section{Definitions and setup}

We refer the reader to \cite{P} for the details and summarize the results here.

Let $F$ be a totally real field, $[F:\mathbb{Q}] = d$, where $d$ is
even with $D$ a totally definite quaternion algebra over $F$,
unramified at all finite places and fix $O_{D}$ to be a maximal
order of $D$. Fix an isomorphism $D \otimes_F K \cong M_2
(K)$, where $K$ is a Galois extension of $\mathbb{Q}$, which splits $D$,
with $F \subseteq K$.

Fix $\mathbf{k} = (k_{\tau}) \in \mathbb{Z}^{I}$ such that each component $k_{\tau}$
is $\geq 2$ and all components have the same parity. Set
$\mathbf{t} = (1,1,...,1) \in Z^{I}$ and set $\mathbf{m} = \mathbf{k} - 2\mathbf{t}$. 
Also choose $\mathbf{v} \in
Z^{I}$ such that each $v_{\tau} \geq 0$, some $v_{\tau} = 0$ and $\mathbf{m}
+ 2\mathbf{v} = \mu \mathbf{t}$ for some $\mu \in \mathbb{Z}_{\geq 0}$.   Let $\mathbb{A}$ be the ring of adeles and $G = Res_{F/\mathbb{Q}} D^{*}$ the algebraic group defined by restriction of scalars. 

Let $R$ be any commutative ring. For any $R$-algebra $A$ and for
$a,b \in \mathbb{Z}_{\geq 0}$, we let $S_{a,b}(A)$ denote the $M_2(R)$-module
$Symm^a(A^2)$ (the $a^{th}$ symmetric power) with an action by $M_2(R)$ given by $x\alpha = (det \alpha)^{b}xS^{a}(\alpha)$. If $A^2$ has a natural basis $e_1, e_2$, then $S_{a,b}(A)$ has a
basis $f_0,...,f_a$ where each $f_i = e_1^{\otimes i} \otimes
e_2^{\otimes (a-i)}$
If $\mathbf{k} \in \mathbb{Z}[I]$ and $\mathbf{m}, 
\mathbf{v}, \mu$ are as before we set 
$L_{\mathbf{k}} = \otimes_{\tau \in I} S_{m_\tau, v_\tau }(\mathbb{C})$. If $R$ is a ring such that $O_{K,v} \subseteq R$, for some $v | p$,
then,  $L_{\mathbf{k}}(R) = \otimes_{\tau \in I} S_{m_{\tau},v_{\tau}}(R) $

We let $p$ be a rational prime which is inert in $K$, $M$ be the semigroup in $M_2(O_{F,p})$ consisting of matrices
$\bp a &b\\c&d\ep$ such that $c \equiv 0 \mod p$ and $d \equiv 1
\mod p$.  Let $U \subseteq G_{f}$ (where $G_f = G(\mathbb{A}_f)$ with $\mathbb{A}_f$ the finite adeles) be
an open compact subgroup such that the projection to $G(F_p)$ lies
inside $M$. If $u \in U$, let $u_{p} \in G(F_{p})$  denote the image under the
projection map.

If $ f: G(\mathbb{A}) \rightarrow L_{\mathbf{k}}(R)$ and $u =
u_{f}.u_{\infty} \in G(\mathbb{A})$ then the weight $\mathbf{k}$ operator is defined as:

$(f|_{\mathbf{k}}u)(x) = u_{\infty}f(x.u^{-1})$, when $R = \mathbb{C}$.

$(f||_{\mathbf{k}}u)(x) = u_{p}f(x.u^{-1})$, when $R$ is an $O_{K,p}$-algebra.

Using the definition of the weight $\mathbf{k}$ operator we can define the space of automorphic forms for $D$, of level $U$ and weight $\mathbf{k}$ as:

$\mathcal{S}_{\mathbf{k}}^D(U) = \{ f: D^{*} \setminus G(\mathbb{A})
\rightarrow L_{\mathbf{k}} \mid f |_{\mathbf{k}}u = f , \forall u \in U\}$ 

$= \{f: G_{f}/U \rightarrow L_{\mathbf{k}} \mid f(\alpha.x) = \alpha.f(x),
\forall  \alpha \in D^{*} \}$

$\mathcal{S}^{D}_{\mathbf{k}}(U, R) = \{ f: D^{*} \setminus G(\mathbb{A})
\rightarrow L_{k}(R) \mid  f ||_{\mathbf{k}} u = f , \forall u \in U \}$.

The purpose of introducing $\mathcal{S}^{D}_{\mathbf{k}}(U, R)$ is to give
$\mathcal{S}^{D}_{\mathbf{k}}(U)$ an integral structure which allows us to think of $\mathcal{S}^{D}_{\mathbf{k}}(U, R)$ as $\oplus_{\gamma_i \in X(U)} (\gamma_i  L_k(R))^{D^{*} \cap \gamma_i
U \gamma_i^{-1}}$. Thus, we see that $S_{\mathbf{k}}^{D}(U,R)$ is an $R$-lattice in
$S_{\mathbf{k}}^{D}(U)$.

Let $X(U) = D^{*}\backslash G_{f} / U$. We know this is finite, so let $h = |X(U)|$
and let $\{\gamma_{i} \}_{i =1}^{h}$ be the coset
representatives. Hence we can write $G_{f} = \coprod_{i = 1}^{h} D^{*}.\gamma_{i}.U$.

Define $\widetilde{\Gamma^{i}(U)} := D^{*} \cap
\gamma_{i}.U.G_{\infty}^{D}.\gamma_{i}^{-1}$ and let $\Gamma^{i}(U) := \widetilde{\Gamma^{i}(U)}/\widetilde{\Gamma^{i}(U)} \cap F^{*}$.

Due to a result by Hida \cite{H1} (Sec $7$), $U$ can be chosen such that the
${\Gamma^{i}(U)}$ are torsion free for all $i$. Coupled
with the statements above, this means that the ${\Gamma^{i}(U)}$ are trivial, provided $U$ is chosen
carefully.

\section{Results}

We follow the description in \cite{B}. Let $L$ be a $R[\xi,\psi]$-module which is  finite and free over $R$, $K$ a submodule of $L$ of finite index so that $L/K \cong \oplus_{i=1}^r O/ p^{a_i} O$ ($a_i \leq n$) and $\xi(K) \subset p^n L$. Let $L'$ be another $R[\xi,\psi]$-module with $K'$ a submodule of finite index such that $\xi(K') \subset p^n L'$ and $L/K \cong L'/K'$ as $R[\xi,\psi]$-modules.  

Define $B(j) = \sum_{i=1}^{j} b(i)$ where $b_i = n - a_i$, and $T(j) = M + B(j-1)$, where $M$ is the smallest integer such that $2M \geq n$. These functions $B$ and $T$ can be thought of as the top and bottom boundaries of the Newton Polygon of $\xi$ acting on $L$.

Let $c(L/K) = \inf_{i \geq 0} \{T(i)/i\}$. This will be the slope of the largest line through the origin that does not lie above $T$. Assume that  for all integers $n'$ with $n - 2\alpha \leq n' \leq n$ we have $\alpha < c(L/ (K + p^{n'} L))$. Let $\kappa$ be any positive integer such that $\kappa \leq n - 2\alpha$ and $\alpha < c(L/ (K + p^{n'} L))$ for $n'$ where $n - 2\alpha - \kappa < n' \leq n $.

Let $p^{\alpha}u, p^{\alpha}u'$ be the roots of slope $\alpha$ of the characteristic polynomials of $\xi$ acting on $L,L'$. Choose $F \in L$ but $F \notin pL$ such that $\xi (F) = p^{\alpha}uF$ and $F' \in L'$  ($F' \notin p L'$) using the isomorphism of $L/K \cong L'/K'$ such that $\xi (F') = p^{\alpha'}u' F'$.  Since $\xi$ and $\psi$ commute, $F,F'$ are eigenvectors for $\psi$ as well. Let $a$ be the eigenvalue of $\psi$ acting on $F$ and $a'$ the eigenvalue of $\psi$ on $F'$. In \cite{B}, he proved the following proposition. 

\begin{proposition}

a and a' are congruent modulo $p^{\kappa}$.

\end{proposition}

\begin{proof} 
We give a sketch of Buzzard's proof.

Let $H \in L$ such that $H + K$ maps to $F' +K'$ under the isomorphism $L/K \cong L'/K'$. Consider the lattice $\Lambda = (\mathbb{Q}_p F + \mathbb{Q}_p \xi H) \cap L $. If it has rank one, then $\xi H$ is a multiple of $F$ which means  $\psi \xi H = a \xi H \Rightarrow (a-a')$  $\xi H \in p^n L \Rightarrow p^{n-2\alpha} | (a-a') \Rightarrow a \equiv a' \mod p^{\kappa}$ as $\kappa \leq n - 2\alpha$. 

So we can assume $\Lambda$ has rank $2$ and extend $F$ to a basis $\{F,G\}$ of $\Lambda$ and write $\xi H = \lambda F + \mu G$ with $\lambda, \mu \in \mathbb{Z}_p$. Let $\mu = p^{\beta}u''$. There are two cases:

$\beta \geq 2 \alpha + \kappa$:

Applying $\psi$ to $\xi H = \lambda F + \mu G$, we get that the lines $\lambda a' F + \mu a' G , \lambda a F + \mu \psi G $ are congruent $\mod p^{n}$. Reducing $\mod p^{2\alpha + \kappa}$, we get that $(a-a')\lambda F \in p^{2 \alpha +\kappa}$. We assumed that $F \notin pL$ so we see that $a \equiv a' \mod p^{\kappa}$. 

$\beta < 2\alpha + \kappa$:

By the definition of $\Lambda$, $L/\Lambda$ is torsion-free so we can extend ${F,G}$ to a basis $\{F,G,l_1,..l_r\}$ of $L$. Define $\tilde{\xi}: L \rightarrow L$ by $\tilde{\xi}(F) = p^{\alpha}uF$, $\tilde{\xi}(G) = \nu F + p^{\alpha}u'G$ and $\tilde{\xi}(l_i) = \xi l_i$ for all $i$ so we see that $\tilde{\xi}$ has $2$ eigenvalues of slope $\alpha$ .  Let $n' = n -\beta$, then ${\xi}(K +p^{n'}L) \subseteq p^{n'}L \subseteq K$ and $\tilde{\xi}$ and $ \xi$ are congruent $\mod p^{n'}$. If $\alpha < c(L/(K + p^{n'}L))$ , then due to the result on local constancy the number of eigenvalues of $\xi$ and $ \tilde{\xi}$ of slope $\alpha$ must be equal. But $\xi$ only has one eigenvalue of slope $\alpha$ so we have a contradiction as we chose $\alpha < c(L/(K + p^{n'}L))$ for $\kappa$ where $n - 2 \alpha - \kappa < n' \leq n$.

\end{proof}

Now we adapt Buzzard's method to the case of Hilbert modular forms. We fix a prime $p \in \mathbb{Q}$ which is inert in $K$, let $R = O_{K,p} $ and assume that $\Gamma^i (U)$ are trivial. 

In order to do some computations on the space of automorphic forms we can think of 

$L_{\mathbf{k}}(R) = L_{k_{1}}(R) \otimes L_{k_{2}}(R)
\otimes ...... \otimes L_{k_{d}}(R)\otimes det()^{n_k}$, where $det()^{n_k}$ accounts for the twist by determinants and
$L_{k_i}(R)$ are the $k_i$th symmetric powers.

 We define $W_{\mathbf{k}}^{n}(R)$ (where $n \leq k_i$)to be generated by the submodules  
 
 $W_{k_{1}}^{n}(R) \otimes L_{k_{2}}(R) \otimes ...... \otimes
L_{k_{d}}(R)\otimes det()^{n_k}$ , $L_{k_{1}}(R) \otimes
W_{k_{2}}^{n}(R) \otimes ...... \otimes L_{k_{d}}(R)\otimes
det()^{n_k}$ and up to $L_{k_{1}}(R) \otimes L_{k_{2}}(R) \otimes ...... \otimes
W_{k_{d}}(R)\otimes det()^{n_k}$, where each $W_{k_{i}}^{n}(R)$ is generated by the $(n+1)$ $R$
submodules $\{ p^{n-j} x^{j} L_{k-j}(R) \}_{j=0}^{n}$. 

Let $L$ correspond to $\mathcal{S}_{\mathbf{k}}^{D}(U,R)$,  $K$ to $\mathcal{W}_{\mathbf{k}}^{D}(U,R)$ equipped with an $R$-linear endomorphism $\xi$ (which corresponds to the $U_p$ operator) and $\psi$ as the $t$-operator. In \cite{P}, by the definition of $\mathcal{W}_{\mathbf{k}}^{D}(U,R)$ and the choice of the $U_p$ operator, we know that $\xi(K) \subset p^n L$ and

$L/K = \mathcal{S}^{D}_{\mathbf{k}}(U,R)/ \mathcal{W}_{\mathbf{k}}^{D}(U,R) \cong
\oplus_{i=1}^{h}L_{\mathbf{k}}(R)/\oplus_{i=1}^{h}W^{n}_{\mathbf{k}}(R) \cong \underbrace{\oplus (O/p^{n}O)}_{\sigma_1 \
\hbox{\scriptsize times}} \underbrace {\oplus
(O/p^{n-1}O)}_{\sigma_2 \ \hbox{\scriptsize times}} .....
\underbrace{\oplus (O/p^{1}O))}_{\sigma_{n} \ \hbox{\scriptsize
times}} \cong \oplus_{i=1}^r O/p^{a_{i}}O$, for all
$\mathbf{k}=(k_{1},k_{2},...,k_{d})$ and where $a_1 \geq a_2 \geq
...$.

To use the proposition above we first need to compute $c(L/(K + p^{n'}L))$.

In \cite{P}, based on the structure of $L/K$ we saw that $B(x)$ was a piecewise linear function with slope $r$ for $r^d h \leq x \leq (r+1)^d h$. We chose a polynomial $q(x) = (x/h)^{1/d} - 1 < B'(x)$ and then computed $Q(x) = \int q(y)dy$ so that $Q(x) < B(x)$ and $P(x): = M + Q(x) < T(x)$. Computing the minimum of $P(x)/x$ we saw that $ c_1 n^{1/(d+1)} - 1 < P(x)/x < T(x)/x$ and found the value of  $c(L/K) = \inf \{T(x)/x\}  = min \{ c_1 n^{1/(d+1)}, n \}$, where $c_1 = (\frac{1}{d+1})^{d/(d+1)}(\frac{1}{h^{d/(d+1)}} + 1)$. 
 
We see that $L/(K + p^{n'}L) \cong \underbrace{\oplus (O/p^{n'}O)}_{\sigma_1 + \sigma_2 +...+ \sigma_{\beta +1} \
\hbox{\scriptsize times}} \underbrace {\oplus
(O/p^{n'-1}O)}_{\sigma_{\beta} \ \hbox{\scriptsize times}} .....
\underbrace{\oplus (O/p^{1}O))}_{\sigma_{n} \ \hbox{\scriptsize
times}}$, where $\beta = n - n'$.

Using the same methods as in \cite{P}, we see that $T(x)$ is a piecewise linear function with slope $r$ for $(r+\beta)^d h \leq x \leq (r + \beta + 1)^d h$.  Thus  $c(L/(K + p^{n'}L) =  c_1 n^{1/(d+1)} - 1 - \beta$, where $c_1$ is as above. 

 In \cite{P} the following theorem was proved:

\begin{theorem}
Let $D(\mathbf{k},\alpha)$ be the number of eigenvalues of the $p^{- \sum
v_i}T_{p}$ operator acting on $\mathcal{S}^{D}_{\mathbf{k}}(U,R)$.
Let $ \alpha \leq c_1 n^{1/(d+1)} + c_2 \Rightarrow n \geq \lfloor
(\beta_{1} \alpha - \beta_{2})^{d+1} \rfloor = n(\alpha)$. If $\mathbf{k}, \mathbf{k'} \geq n(\alpha)$, $\mathbf{k} \equiv \mathbf{k'} \mod p^{n(\alpha)}$ and $\gamma^i$ are trivial, then $D(\mathbf{k},\alpha) =
D(\mathbf{k'},\alpha) $. (The $\beta_i$ are constants which depend only on $\alpha$ and $n$)
\end{theorem}

We can now state and prove our main theorem.

\begin{theorem}
Assuming the conditions of the previous theorem, let us further assume that there exists $d_0$ such that for all $\mathbf{k}$ where $\mathbf{k} \equiv \mathbf{k_0} \mod p^{n(\alpha)}$ we have $D(\mathbf{k},\alpha)= d_0$. If $d_o = 1$ let $F_k$ denote the unique up to scalar form of slope $\alpha$ which is an eigenform for all $t \in \mathbb{T}$ with $a_k(t)$ the eigenvalue of $t$ on $F_k$. Choose $\kappa = \lfloor  c_1 n^{1/d+1} -1 - 3 \alpha \rfloor$. Then if $n > (\frac{\kappa + 1 + 3\alpha}{c_1})^{d+1}$and
$\mathbf{k},\mathbf{k'} \geq n(\alpha)$ congruent to $\mathbf{k}_0$ and $\mathbf{k} \equiv \mathbf{k'} \mod p^{n -1}$  we have for all $t \in \mathbb{T}$, $a_{\mathbf{k}} (t) \equiv a_{\mathbf{k'}} (t) \mod p^{\kappa}$.

\end{theorem}

\begin{proof}

To prove the theorem we will use Buzzard's proposition. Thus, we need to verify that $2\alpha + \kappa \leq n$ and that for all integers $n'$ where $n - 2\alpha - \kappa < n' \leq n$ we have that $\alpha < c(L/ (K+ p^{n'}L))$.

To show that $2\alpha + \kappa \leq n$, we know that 

$\kappa < c_1 n^{1/d+1} -1 -3\alpha \Rightarrow 2\alpha +\kappa < c_1 n^{1/d+1} -1 -\alpha \Rightarrow 2\alpha +\kappa < c_1n^{1/d+1}$, where $c_1 = (\frac{1}{d+1})^{d/(d+1)}(\frac{1}{h^{d/(d+1)}} + 1)$.  

Now $c_1n^{1/d+1} =  (\frac{1}{d+1})^{d/d+1}(\frac{1}{h^{d/d+1}} +1)n^{1/d+1} \leq (\frac{1}{(d+1)h} +\frac{1}{d+1})^{d/d+1}n^{1/d+1} < n^{1/d+1} <  n$. 

For the second assertion, let $\beta = n - n'$, so we need to show that $\beta < 2\alpha +\kappa$. 

$\beta < 2\alpha +\kappa \Rightarrow \alpha < c(L/ (K+ p^{n'}L)) =   c_1 n^{1/(d+1)} - 1 - \beta$. 

We look at the choice of $\kappa =  \lfloor  c_1 n^{1/d+1} -1 - 3 \alpha \rfloor$. 

If $\beta < 2\alpha + \kappa $, this means that $\beta <  c_1 n^{1/d+1} -1 -  \alpha \Rightarrow  \alpha <  c_1 n^{1/d+1} -1 -  \beta $, so we are done.

\end{proof}

\section{Concluding Remarks}

\begin{itemize}
\item In \cite{B}, the result was established for classical modular forms ($d=1$). Substituting $d=1$ in our calculations would give us the bound for $n$ in terms of $\alpha^2$ in place of $\alpha^{d+1}$.

\item Hida \cite{H1} provided examples of families of Hilbert Modular forms of slope zero (or ordinary Hilbert Modular forms).

\item The existence of families of non ordinary Hilbert Modular forms is due to \cite{KL} and the results we obtain here are much weaker as we prove only continuity, not analyticity, of the families.  See also Yamagami \cite{Y} for similar results using rigid analytic methods.

\end{itemize}

Aftab Pande,

Universidade Federal do Rio de Janeiro,

Rio de Janeiro,

Brasil.

aftab.pande@gmail.com


\begin{thebibliography}{8}

\bibitem[B]{B}
Buzzard, K.  p-adic modular forms on definite quaternion algebras, unpublished. 

(http://www2.imperial.ac.uk/~buzzard/maths/research/notes/definite.dvi)

\bibitem[BC]{BC}
Buzzard, K; Calegari, F. A counterexample to the Gouvea-Mazur conjecture, {\em C. R. Math. Acad. Sci. Paris} 338 (2004), $\mathbf{10}$, 751 - 753. 

\bibitem[C]{C}
Coleman, R.  $p$-adic Banach spaces and families of modular
forms, {\em Invent. Math} 127, $\mathbf{3}$, (1997) 417 - 479.

\bibitem[CM]{CM}
Coleman, R; Mazur, M. The eigencurve, {\em Galois
Representations in arithmetic algebraic geometry}, Durham (1996), CUP
1998, 1 - 113.

\bibitem[GM]{GM}
Gouvea, F;  Mazur, B.  Families of Modular Eigenforms, {\em Math.
Comp}. Vol. 58 $\mathbf{198}$ (1992), 793 - 805.

\bibitem[H1]{H1}
Hida, H. On $p$-adic Hecke Algebras for $GL_2$ over Totally Real
Fields, {\em Ann Math}, $\mathbf{128}$ (1988), 295 - 384.

\bibitem[H2]{H2}
Hida, H. Iwasawa modules attached to congruences of cusp forms, {\em Ann. Sci. Ecole Norm. Sup}. (4) 19 (1986), $\mathbf{2}$, 231 - 273.

\bibitem[H3]{H3}
Hida, H. Galois representations into $GL_2 (\mathbb{Z}_p[[X]])$
attached to ordinary cusp forms,  {\em Invent. Math}. 85 (1986), $\mathbf{3}$,
545 - 613.

\bibitem[KL]{KL}
Kisin, M; Lai, K.F.  Overconvergent Hilbert modular forms, {\em American Journal of Mathematics} 
Volume 127, $\mathbf{4}$, August 2005 pp. 735-783.

\bibitem[P]{P}
Pande, A. Local constancy of dimensions of hecke eigenspaces of automorphic forms, {\em J. Number Theory} 129 (2009), $\mathbf{1}$, 15 - 27.

\bibitem[S]{S} 
Serre, J.P. Formes modulaires et fonctions zeta
$p$-adiques, Modular forms in one variable III,{\em Lecture Notes in
Mathematics}, $\mathbf{350}$, Springer Verlag, 1973.

\bibitem[Y]{Y} 
Yamagami, A.On $p$-adic families of Hilbert cusp forms of finite slope.  {\em J. Number Theory}  123  (2007), $\mathbf{2}$, 363 - 387.
 
\end{thebibliography}
\end{document}